\documentclass[reqno,11pt]{amsart}
 
\usepackage{amsmath}
\usepackage{amsthm}  
\usepackage{verbatim}
\usepackage{enumerate} 
\usepackage{mathtools}  
\usepackage{amssymb}
\usepackage{mathrsfs}
\usepackage[top=1.3in, bottom=1.0in, left=0.7in, right=0.7in]{geometry}  
\usepackage[backref=page, colorlinks]{hyperref}
\hypersetup{
	colorlinks,
	citecolor=blue,
	linkcolor=red
}   
\numberwithin{equation}{section}
\usepackage{xypic}
\usepackage{bm}
\usepackage{tikz}
\usetikzlibrary{decorations.pathreplacing}
\usepackage{xcolor}
\usepackage{float}
\usepackage{cleveref}
\usepackage{comment}
\usepackage[abs]{overpic}

\newtheorem{theorem}{\textbf{Theorem}}[section]
\newtheorem{theorem*}{\textbf{Theorem}}

\newtheorem{proposition}[theorem]{\textbf{Proposition}}
\newtheorem{lemma}[theorem]{\textbf{Lemma}}
\newtheorem{question}[theorem]{Question}

\newtheorem{corollary}[theorem]{\textbf{Corollary}}
\newtheorem{remark}[theorem]{\textbf{Remark}}

\newtheorem{example}[theorem]{\textbf{Example}}

\newtheorem{definition/proposition}[theorem]{\textbf{Definition/Proposition}}

\newcommand{\Addresses}{{

\bigskip
		\footnotesize

        Youlin Li, \par\nopagebreak
	    \textsc{School of Mathematical Sciences, Shanghai Jiao Tong University, Shanghai 200240, China}\par\nopagebreak
		\textit{E-mail address}: \href{liyoulin@sjtu.edu.cn}{liyoulin@sjtu.edu.cn}

		\bigskip
		\footnotesize

        Chi Zhang, \par\nopagebreak
	    \textsc{School of Mathematical Sciences, Shanghai Jiao Tong University, Shanghai 200240, China}\par\nopagebreak
		\textit{E-mail address}: \href{2460889966651@sjtu.edu.cn}{2460889966651@sjtu.edu.cn}

}}
\title{Characterizing slopes for Legendrian knots}
\author{Youlin Li and Chi Zhang}
\begin{document}
	\maketitle
    
\begin{abstract} 
We establish a criterion relating smooth and contact characterizing slopes under a uniqueness assumption. Let $L$ be a Legendrian representative of a knot $K\subset S^3$ with standard contact structure, and assume that the isotopy class of $L$ is uniquely determined by its classical invariants: the Thurston--Bennequin invariant $tb(L)$ and the rotation number. Then, for any non-zero rational number $r$, if $r+tb(L)$ is a smooth characterizing slope for $K$, it becomes a contact characterizing slope for $L$. As applications, we study the characterizing slopes for Legendrian representatives of the unknot, trefoil, figure-eight knot, cinquefoil, $5_2$, and $\overline{5_2}$.
\medskip

\end{abstract}
\section{Introduction}
A slope $\frac{p}{q}$ is called a \textit{characterizing slope} for a given knot $K$ in $S^3$ if whenever the $\frac{p}{q}$–surgery on a knot $K'$ in $S^3$ is homeomorphic to the $\frac{p}{q}$–surgery on $K$ via an orientation preserving homeomorphism, then $K'$ is isotopic to $K$. By the results in \cite{KMOS} and \cite{OSZ19}, any rational number is characterizing for the unknot, the trefoils and figure eight knot. In \cite{NiZh14, NiZh23}, Ni and Zhang showed that not all nontrivial slopes for a torus knot are characterizing slopes. However, if $r>s>1$, they showed that a nontrivial slope $\frac{p}{q}$ is a characterizing slope for the torus knot $T_{r,s}$ if $\frac{p}{q} > \frac{30(r^{2}-1)(s^{2}-1)}{67}$. They obtain more specific information about the characterizing slopes of the positive cinquefoil $T_{5,2}$. In \cite{BaSi24}, Baldwin and Sivek proved that any rational number, other than a positive integer, is characterizing for the knot $5_2$.

Contact Dehn surgery along Legendrian links is a very efficient way to construct contact 3-manifolds \cite{Geiges}. Let $L$ be a Legendrian knot in $(S^3, \xi_{st})$, the contact $r$-surgery on $L$, with $r\in \mathbb{Q}\setminus\{0\}$, is the smooth 3-manifold obtained from $S^3$ by topological $r+tb(L)$ surgery on $L$, and extending the restriction of $\xi_{st}$ on the complement of the standard neighborhood of $L$ by a tight contact structure on the surgery torus. In this paper, we denote the resulted contact 3-manifolds by $L(r)$. 
If $r=\frac{1}{q}$, then the resulted contact 3-manifold $L(\frac{1}{q})$ is uniquely determined. However, in the general case, $L(r)$ is not unique. Instead, it is a finite collection of contact 3-manifolds with the same underlying smooth closed 3-manifold.

The collection $L(r)$ is said to be contactomorphic to the collection $L'(r')$ if there exists a bijection between the elements in $L(r)$ and $L'(r')$ such that paired elements are contactomorphic. A contact surgery slope $r\in \mathbb{Q}$ is called a \textit{characterizing slope} if whenever the collection $L(r)$ is contactomorphic to $L'(r)$ for some Legendrian knot $L'$, then $L'$ is Legendrian isotopic to $L$ in $(S^3,\xi_{st})$. Since the contact surgery along a Legendrian knot is independent of its orientation, we only consider the unoriented Legengendrian knots in this paper, and assume the rotation numbers of all Legendrian knots are positive.

In \cite{Et08}, Etnyre proved that the contact $(+1)$-surgeries along a pair of distinct Legendrian knots with the same classical invariants often produce contactomorphic contact manifolds. So $+1$ is not a characterizing slope for these Legendrian non-simple knots.

In \cite{CEK}, Casals, Etnyre and Kegel constructed infinitely many pairs of distinct Legendrian knots in $(S^3, \xi_{st})$ so that the contact $(-1)$-surgeries along any pair of them yield the same Stein traces, in particular, the same contact 3-manifolds. That means $-1$ is not a characterizing slope for these Legendrian knots. They also obtain some characterizing slopes for Legendrian simple knots, including Legendrian unknots, Legendrian trefoils, Legendrian figure eight knots, and some hyperbolic knots. See \cite[Theorems 1.8, 1.10]{CEK}.

In this paper, we give a criterion for a non-zero rational number to be a contact characterzing slope for a Legedrian knot.

\begin{theorem}\label{Thm:anyknot}
Suppose $K\subset S^3$ is a knot and $L$ is a Legendrian representative of $K$ in $(S^3, \xi_{st})$. If a non-zero rational number $\frac{p}{q}$ satisfies that $\frac{p}{q}+tb(L)$ is a (smooth) characterizing slope for $K$, and $L(\frac{p}{q})$ is contactomorphic to $L'(\frac{p}{q})$ for some Legendrian knot $L'$ in $(S^3, \xi_{st})$, then $L'$ is smoothly isotopic to $L$, $tb(L')=tb(L)$ and $rot(L')=rot(L)$.
\end{theorem}

Immediately, we have

\begin{corollary}\label{Thm:anyknot1}
Suppose $K\subset S^3$ is a knot and $L$ is a Legendrian representative of $K$ in $(S^3, \xi_{st})$. If the isotopy class of $L$ is uniquely determined by its Thurston-Bennequin invariant and rotation number, then any non-zero rational number $\frac{p}{q}$ such that $\frac{p}{q}+tb(L)$ is a (smooth) characterizing slope for $K$ is a (contact) characterizing slope for $L$.
\end{corollary}

As applications, we can show that any non-zero rational number is a characterizing slope for any Legendrian representatives of unknot, trefoil and figure eight knot.

\begin{corollary}\label{Thm:unknot}
Let $L\subset (S^3,\xi_{st})$ be a Legendrian representative of the unknot,  the right handed trefoil, the left handed trefoil, or the figure eight knot. Then any non-zero rational number $\frac{p}{q}$ is a characterizing slope for $L$.
\end{corollary}

This answers \cite[Question 5.4]{CEK}. We also show that many rational numbers are characterizing slopes for Legendrian representatives of the knots $5_2$, $\overline{5_2}$, $T_{5,-2}$ and $T_{5,2}$.

\begin{corollary}\label{Thm:52} We determine partial characterizing slopes for Legendrian representatives of  $5_2$ and $\overline{5_2}$.\\
(1) Let $L\subset (S^3,\xi_{st})$ be a Legendrian representative of $5_2$. Then any non-integral rational number $\frac{p}{q}$ is a characterizing slope for $L$. If $p$ is a non-zero integer satisfying that $p+tb(L)\leq 0$, then $p$ is a characterizing slope for $L$. \\
(2) Let $L\subset (S^3,\xi_{st})$ be a Legendrian representative of $\overline{5_2}$ with non-maximal Thurston-Bennequin invariant. Then any non-integral rational number $\frac{p}{q}$ is a characterizing slope for $L$. If $p$ is a non-zero integer satisfying that $p+tb(L)\geq 0$, then $p$ is a characterizing slope for $L$.\\
(3) Let $L\subset (S^3,\xi_{st})$ be a Legendrian representative of $\overline{5_2}$ with maximal Thurston-Bennequin invariant. Then $-1$ is a characterizing slope for $L$. 
\end{corollary}

\begin{corollary}\label{Thm:T52} We determine partial characterizing slopes for Legendrian representatives of  $T_{5,-2}$ and $T_{5,2}$.\\
(1) Let $L\subset (S^3,\xi_{st})$ be a Legendrian representative of $T_{5,-2}$. If $\frac{p}{q}+tb(L)<1$ and $\frac{p}{q}+tb(L)\notin\{0,-1,\pm\frac{1}{2}, \pm\frac{1}{3}\}$, then $\frac{p}{q}$ is a characterizing slope for $L$. \\ 
(2) Let $L\subset (S^3,\xi_{st})$ be a Legendrian representative of $T_{5,2}$. If $\frac{p}{q}+tb(L)>-1$ and $\frac{p}{q}+tb(L)\notin\{0,1,\pm\frac{1}{2}, \pm\frac{1}{3}\}$, then $\frac{p}{q}$ is a characterizing slope for $L$. 
\end{corollary}

We give two examples to show that not every non-zero rational is characterizing for a general Legendrian torus knot or $\overline{5_2}$.

\begin{example}
For a general Legendrian knot $\overline{5_2}$, not every non-zero integer is a characterizing slope. By \cite[Proposition 1.2]{BaSi24}, $S^{3}_{-1}(\overline{5_2})$ is orientation-preserving homeomorphic to $S^{3}_{-1}(P(3,-3,-8))$. Let $L_1$ be a Legendrian knot $\overline{5_2}$ with Thurston-Bennequin invariant $1$ and rotation number $0$, and $L_2$ be a Legendrian pretzel knot $P(3,-3,-8)$ with Thurston-Bennequin invariant $-3$ and rotation number $2$. See Figure~\ref{fig:noncharacterizing slope}. Then the contact $(+3)$-surgeries along $S_{+}^{4}S_{-}(L_1)$ and along $S_{+}(L_2)$ both yield overtwisted contact 3-manifolds \cite[Theorem 1.1]{Co19}, where $S_{+}$ and $S_{-}$ denote the positive and negative stabilizations, respectively. We have $tb(S_{+}^{4}S_{-}(L_1))=tb(S_{+}(L_2))=-4$ and $rot(S_{+}^{4}S_{-}(L_1))=rot(S_{+}(L_2))=3$. Since the resulted two groups of overtwisted contact 3-manifolds are diffeomorphic and have the same $d_3$-invariants, they are contactomorphic. So $+3$ is not a characterzing slope of either $S_{+}^{4}S_{-}(L_1)$ or $S_{+}(L_2)$.
\end{example}

\begin{figure}[htb]
\begin{overpic}
[scale=0.5]
{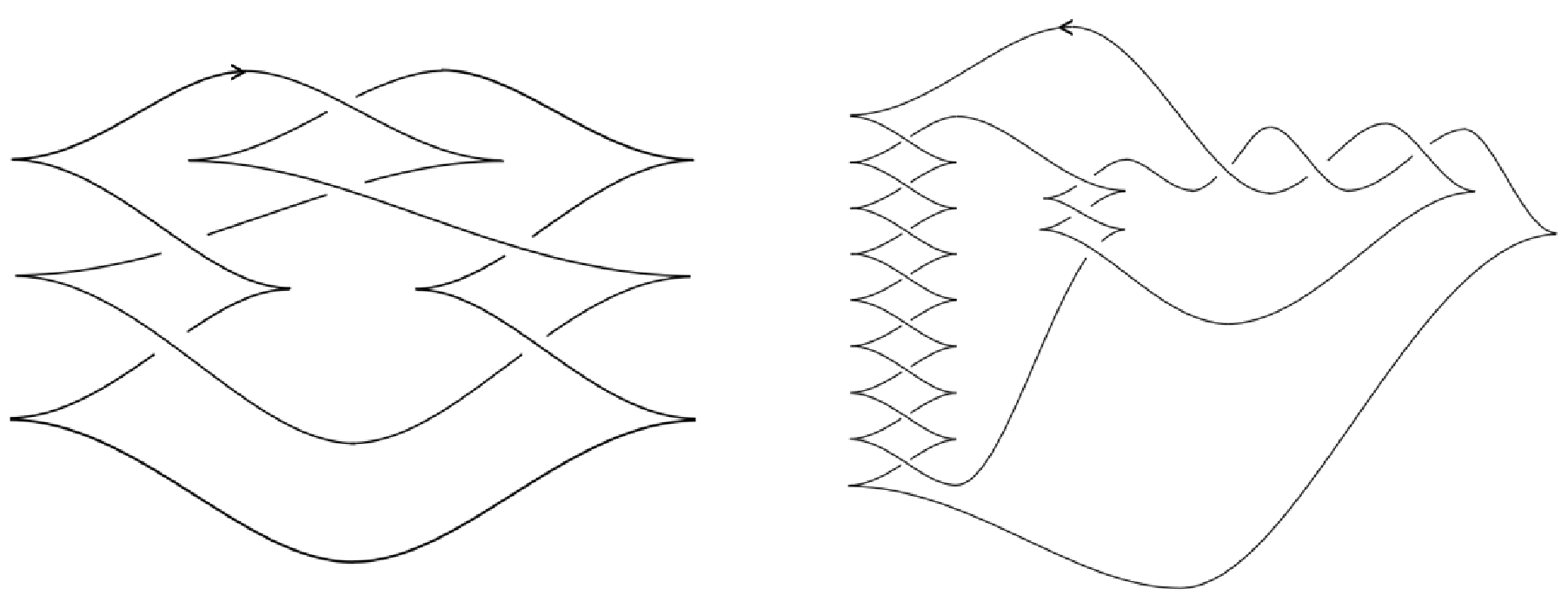} 
\put(30, 20){$L_1$}
\put(370, 20){$L_2$}
\end{overpic}
\caption{Here $L_1$ is a Legendrian $\overline{5_2}$ with $tb(L_1)=1$ and $rot(L_1)=0$, $L_2$ is a Legendrian $P(3,-3,-8)$ with $tb(L_2)=-3$ and $rot(L_2)=2$.}
\label{fig:noncharacterizing slope}
\end{figure}

\begin{example}
For a general Legendrian torus knot, not every non-zero rational number is a characterizing slope. For example, Let $L_1$ be a Legendrian torus knot $T_{-5,4}$ with Thurston-Bennequin invariant $-20$ and rotation number $1$, and $L_2$ be a Legendrian torus knot $T_{-11,2}$ with Thurston-Bennequin invariant $-22$ and rotation number $1$. Then the contact $(+1)$-surgeries along $S_{+}S_{-}(L_1)$ and along $L_2$ both yield overtwisted contact 3-manifolds \cite[Corollary 1.2]{LiSt06}. Since these two resulted contact 3-manifolds are diffeomorphic \cite{NiZh14} and have the same $d_3$-invariants, they are contactomorphic. So $+1$ is not a characterzing slope of either $S_{+}S_{-}(L_1)$ or $L_2$.
\end{example}

\subsection*{Acknowledgments}
We thank Zhongtao Wu for stimulating conversations. We also thank Marc Kegel for useful comments on an earlier draft of the paper. Both authors were partially supported by the National Natural Science Foundation of China under Grant No. 12271349. 
\bigskip

\section{Preliminaries}
\subsection{Topological surgery.}
For a knot $K$ in $S^3$ and $r\in\mathbb{Q}$, we denote the closed 3-manifold obtained by topological $r$-surgery on $S^3$ along a knot $K$ by $S^{3}_{r}(K)$. Let $U$ be the unknot. 
The lens space obtained by $\frac{p}{q}$-surgery along the unknot $U$, $S^{3}_{\frac{p}{q}}(U)$, is denoted by $L(p,-q)$. At first, we recall some important facts about smooth characterizing slopes.


\begin{theorem}\label{Thm:surcharunknot} \cite{KMOS} Assume $U$ is an unknot, and $K$ is any knot in $S^3$. If there exists a rational number $r$, such that there exists an orientation-preserving diffeomorphism $S_r^3(K)\cong S_r^3(U)$, then $K$ is isotopic to $U$.
\end{theorem}

\begin{theorem}\label{Thm:surchartrefoil} \cite{OSZ19}
Let $K_0$ be a trefoil knot or a figure eight knot. If $K$ is a knot with the property that there is a rational number $r$ and an orientation-preserving diffeomorphism
$S^3_r(K)\cong S^3_{r}(K_0)$, then $K$ is isotopic to $K_0$.
\end{theorem}

\begin{theorem}\label{Thm:surcharT52} \cite{NiZh23}
Suppose $r>-1$ is a rational number, and $r\notin\{0, 1, \pm\frac{1}{2}, \pm\frac{1}{3}\}$. If $K$ is a knot with the property that there is an orientation-preserving diffeomorphism
$S^3_r(K)\cong S^3_{r}(T_{5,2})$, then $K$ is isotopic to $T_{5,2}$.
\end{theorem}

\begin{theorem}\label{Thm:surchar52} \cite{BaSi24}
Suppose $r$ is a rational number other than a positive integer. If $K$ is a knot with the property that there is an orientation-preserving diffeomorphism
$S^3_r(K)\cong S^3_{r}(5_2)$, then $K$ is isotopic to $5_2$.
\end{theorem}

We recall the $d$-invariants formula of Ni and Wu.
\begin{proposition}\label{Prop:dinvformula}\cite[Proposition 1.6, Remark 2.10]{NiWu}
Suppose $p, q>0$, and fix $0\leq i\leq p-1$. Then
$$ d(S^{3}_{\frac{p}{q}}(K), i)=d(S^{3}_{\frac{p}{q}}(U), i)-2\max\{V_{\lfloor\frac{i}{q}\rfloor}(K), V_{\lfloor\frac{p+q-1-i}{q}\rfloor}(K)\}.$$
\end{proposition}

\begin{remark}\label{rem:vi}
For each integer $i\geq 0$, $V_i(K)$ is a non-negative integral concordance invariant of $K$. Moreover, by \cite[Proposition 7.6]{Ras}, $V_{i-1}(K)\geq V_{i}(K)\geq V_{i-1}(K)-1$.
\end{remark}

We also need two theorems from number theory.

\begin{theorem}\label{lem:4k+3}\cite[Theorem 82]{HaWr08}
The number $-1$ is a quadratic non-residue of primes of the form $4k+3$, where $k$ is a non-negative integer. That is, for any integer $x$, $x^2+1$ has no prime factor of the form $4k+3$.
\end{theorem}


\begin{theorem}\label{notwosquares}\cite[Theorem 366]{HaWr08}
A positive integer $n$ is a sum of two squares if and only if all prime factors of $n$ of the form $4k+3~(k\in \mathbb{Z}_{\ge0})$ have even exponents in the standard form of $n$. 
\end{theorem}

Applying Proposition~\ref{Prop:dinvformula}, we obtain the following results about integral and half-integral surgeries along knots.

\begin{proposition}\label{Prop:surcharanyknot}
Suppose $K$ and $K'$ are any two knots in $S^3$, $p$ is a non-zero integer. If $S^3_p(K)$ and $S^3_{-p}(K')$ are orientation-preserving homeomorphic, then $p$ has no prime factor of the form $4k+3$.
\end{proposition}

\begin{proof}
Suppose there is an orientation-preserving homeomorphism between $S^{3}_{-p}(K')$ and $S^{3}_{p}(K)$. We may as well assume $p>0$. Then the $d$-invariants $$d(S^{3}_{-p}(K'), i)=d(S^{3}_{p}(K), \sigma(i)),$$ where $i\in\{0,1,\ldots, p-1\}$ the set of $\mathrm{Spin}^{\textit{c}}$ structures, and $\sigma$ is a permutation of the set. 

  Since $$d(S^{3}_{-p}(K'), i)=-d(S^{3}_{p}(\overline{K'}), i),$$ we have 
$$d(S^{3}_{p}({K}), \sigma(i))=-d(S^{3}_{p}(\overline{K'}), i).$$
By Proposition~\ref{Prop:dinvformula}, 
$$d(S^{3}_{p}({K}), i)=d(S^{3}_{p}(U), i)-2\max \{V_{i}({K}), V_{p-i}({K})\},$$ 
$$d(S^{3}_{p}(\overline{K'}), i)=d(S^{3}_{p}(U), i)-2\max \{V_{i}(\overline{K'}), V_{p-i}(\overline{K'})\},$$ where $i=0,1,\ldots, p-1$. Consider the case $i=1$,  we get
$$d(S^3_p(U),1)+d(S^3_p(U),\sigma(1))=2(\max \{V_{1}(\overline{K'}), V_{p-1}(\overline{K'})\}+\max \{V_{\sigma(1)}({K}), V_{p-\sigma(1)}({K})\})\in \mathbb{Z}.$$
We have \begin{align*}
    d(S^3_p(U),1)+d(S^3_p(U),\sigma(1))&=\frac{(p-2)^2}{4p}-\frac{1}{4}+\frac{(p-2\sigma(1))^2}{4p}-\frac{1}{4}\\
&=\frac{p^2-p+2(\sigma(1)^2+1)}{2p}-1-\sigma(1)\\
&=\frac{p-1}{2}+\frac{\sigma(1)^2+1}{p}-1-\sigma(1)\in\mathbb{Z}.
\end{align*}

If $p$ is odd, then $p$ divides $\sigma(1)^2+1$. By Theorem~\ref{lem:4k+3},  $p$ has no prime factor of the form $4k+3$. If $p$ is even, then $\frac{p}{2}$ divides $\sigma(1)^2+1$. So $\frac{p}{2}$ has no prime factor of the form $4k+3$. This further implies $p$ has no prime factor of the form $4k+3$.
\end{proof}


\begin{proposition}\label{Prop:sucharanyknot2}
Suppose $K$ and $K'$ are any two knots, $p$ is an odd integer. If $S^3_{\frac{p}{2}}(K)$ and $S^3_{-\frac{p}{2}}(K')$ are orientation-preserving homeomorphic, then $|p|=4l+1$ for some integer $l$.
\end{proposition}
\begin{proof}
We may as well assume $p\ge1$. The $d$-invariants $$d(S^{3}_{-\frac{p}{2}}(K'), i)=d(S^{3}_{\frac{p}{2}}(K), \sigma(i)),$$ where $i\in\{0,1,\ldots, p-1\}$ the set of $\mathrm{Spin}^{\textit{c}}$ structures, and $\sigma$ is a permutation of the set.    Since $$d(S^{3}_{-\frac{p}{2}}(K'), i)=-d(S^{3}_{\frac{p}{2}}(\overline{K'}), i),$$
we have $$d(S^{3}_{\frac{p}{2}}(K), \sigma(i))=-d(S^{3}_{\frac{p}{2}}(\overline{K'}), i).$$
By Proposition~\ref{Prop:dinvformula}, 
    $$d(S^{3}_{\frac{p}{2}}(K), i)=d(S^{3}_{\frac{p}{2}}(U), i)-2\max \{V_{\lfloor \frac{i}{2}\rfloor}(K), V_{\lfloor \frac{p+1-i}{2}\rfloor}(K)\},$$
    $$d(S^{3}_{\frac{p}{2}}(\overline{K'}), i)=d(S^{3}_{\frac{p}{2}}(U), i)-2\max \{V_{\lfloor \frac{i}{2}\rfloor}(\overline{K'}), V_{\lfloor \frac{p+1-i}{2}\rfloor}(\overline{K'})\}.$$
According to \cite[Proposition 4.8]{OSZgrading} \footnote{In \cite{OSZgrading}, $S^3_{\frac{p}{q}}(U)$ is denoted by $L(p,q)$.}, $$d(S^{3}_{\frac{p}{2}}(U),i)=\frac{(2i-1-p)^2}{8p}-\frac{1+(-1)^i}{4}.$$ It can be obtained that $$d(S^{3}_{\frac{p}{2}}(K),i)=d(S^{3}_{\frac{p}{2}}(K),p+1-i),$$  $$d(S^{3}_{\frac{p}{2}}(\overline{K'}),i)=d(S^{3}_{\frac{p}{2}}(\overline{K'}),p+1-i).$$ For $s\in\{0, 1, \ldots, p-1\}$, let $$I(s)=\{0\le i\le p-1:d(S^{3}_{\frac{p}{2}}(K),i)=d(S^{3}_{\frac{p}{2}}(K),s)\},$$ $$J(s)=\{0\le i\le p-1:d(S^{3}_{\frac{p}{2}}(\overline{K'}),i)=d(S^{3}_{\frac{p}{2}}(\overline{K'}),s)\}.$$ Then both $|I(s)|$ and $|J(s)|$ are even unless $s=\frac{p+1}{2}$. So we have $\sigma(J(\frac{p+1}{2}))=I(\frac{p+1}{2})$, and then
    $$d(S^{3}_{\frac{p}{2}}(K),\frac{p+1}{2})=-d(S^{3}_{\frac{p}{2}}(\overline{K'}),\frac{p+1}{2}).$$
    That is
    $$-\frac{1+(-1)^{\frac{p+1}{2}}}{2}=2(V_{\lfloor\frac{p+1}{4}\rfloor}(K)+V_{\lfloor\frac{p+1}{4}\rfloor}(\overline{K'})).$$
    So $p=4l+1$, $l\in\mathbb{Z}$.
\end{proof}
\bigskip

\subsection{Contact surgery} We recall the algorithm of Ding, Geiges and Stipsicz \cite{ding2004surgery} to convert a contact $\frac{p}{q}$-surgery diagram to a contact $(\pm1)$-surgery diagram.

(1) Contact $\frac{p}{q}$-surgery along a Legendrian knot $L$ with $\frac{p}{q}<0$:

(i) Stabilize $L$ $|a_1+2|$ times, where 
\[
\frac{p}{q} = a_1+1 - \cfrac{1}{a_2 - \cfrac{1}{\cdots - \cfrac{1}{a_{m-1} - \cfrac{1}{a_m}}}},
\]
where $a_i\leq -2$ for $1\leq i\leq m$. Let the resulting Legendrian knot be $L_1$.

(ii) For $i=2, \ldots, n$, let $L_i$ be the Legendrian push-off of $L_{i-1}$ and stabilize it $|a_{i}+2|$ times.

(iii) Then a contact $\frac{p}{q}$-surgery along $L$ corresponds to a contact $(-1)$-surgeries along a link $L_1\sqcup L_2\sqcup\ldots\sqcup L_n$.

(2) Contact $\frac{p}{q}$-surgery along a Legendrian knot $L$ with $\frac{p}{q}>0$:

(i) Choose a positive integer $k$ such that $q-kp<0$. Let $\tilde{r}=\frac{p}{q-kp}$.

(ii) Let $L_1,\ldots, L_k$ be $k$ successive Legendrian push-offs of $L$.

(iii) Then a contact $\frac{p}{q}$-surgery along a Legendrian knot $L$ corresponds to contact $(+1)$-surgeries along $L_1,\ldots, L_k$ and a contact $\tilde{r}$-surgery along $L$.

\begin{lemma}\cite{DiGe01}
Contact $(\pm\frac{1}{n})$-surgery ($n\in\mathbb{N}$) along a Legendrian knot $L$ yields the same contact manifold as contact $(\pm1)$-surgeries along $n$ Legendrian push-offs of $L$.
\end{lemma}

We also recall the formulas of Euler classes and $d_3$-invariants of contact 3-manifolds in \cite{DuKe}, which stems from \cite{ding2004surgery}.
Let $L_1\sqcup \cdots\sqcup L_k\subset (S^3, \xi_{st})$ be an oriented Legendrian link in a contact surgery diagram, with contact surgery coefficients $\pm1/n_i$ on $L_i$, where $n_i\in\mathbb{N}$ and $i=1,\ldots, k$. Write $tb_i$ for the Thurston–Bennequin invariant of $L_i$, $r_i$ for the rotation number of $L_i$, and $l_{ij}$ for the linking number between $L_i$ and $L_j$. Then the topological surgery coefficient on $L_i$ is $\frac{p_{i}}{q_{i}}=\frac{\pm1+n_{i}tb_{i}}{n_{i}}$. We have a \textit{generalized linking matrix}: 
    \begin{align*}
	Q:=\begin{pmatrix}
	p_1&q_2 l_{12} &\cdots&q_k l_{1k}\\
	q_1 l_{21} & p_2&&\\
	\vdots&&\ddots\\
	q_1 l_{k1}&&& p_k
	\end{pmatrix}.
    \end{align*}

\begin{theorem}\label{thm:d3}\cite{DuKe}
Let $L_1 \sqcup \ldots\sqcup  L_k$ be an oriented Legendrian link in $(S^3, \xi_{st})$ and denote by $(M, \xi)$ the contact manifold obtained from $(S^3, \xi_{st})$ by contact $(\pm1/n_i)$-surgeries along this link ($n_i \in \mathbb{N}$).

1. The Poicar\'e-dual of the Euler class is given by
$$PD (e(\xi)) =\sum\limits_{i=1}^{k}n_{i}r_{i} [\mu_i]\in H_1(M),$$ 
 where $[\mu_i], i=1,\ldots k$, are subject to the linear relation $Q\boldsymbol{\mu}=\mathbf{0}$, with $\boldsymbol{\mu}=([\mu_1],\ldots,[\mu_n])^T$. 

2. The Euler class $e(\xi)$ is torsion if and only if there exists a rational
solution $\boldsymbol{b} \in \mathbb{Q}^k$ of $Q\boldsymbol{b} = \boldsymbol{r}$, where $\boldsymbol{r}=(r_1,\ldots, r_k)^T$. In this case, the $d_3$-invariant computes as
$$d_3(\xi)=\frac{1}{4}(\sum\limits_{i=1}^{k}
n_i b_i r_i+(3-n_i) \operatorname{sign}_i)-\frac{3}{4}\sigma(Q)-\frac{1}{2},$$
where $\operatorname{sign}_i$ denotes the sign of the contact surgery coefficient of $L_i$.
\end{theorem}

By \cite[Reamrk 5.2]{DuKe}, all eigenvalues of $Q$ are real. So we can define the signature $\sigma(Q)$ as the difference of the numbers of positive and negative eigenvalues of $Q$. Theorem~\ref{thm:d3} is a generalization of the following algorithm. 

\begin{theorem}\label{thm:d3s}\cite{ding2004surgery}
If $L_1 \sqcup \ldots\sqcup  L_k$ is an oriented Legendrian link in $(S^3, \xi_{st})$ and  $(M, \xi)$ is the contact manifold obtained from $(S^3, \xi_{st})$ by contact $(\pm1)$-surgeries along it with torsion Euler class, 
then
$$d_3(\xi)=\frac{1}{4}(\sum\limits_{i=1}^{k}
b_i r_i+2 \operatorname{sign}_i)-\frac{3}{4}\sigma(Q)-\frac{1}{2}=\frac{1}{4}\boldsymbol{r}^{T}Q^{-1}\boldsymbol{r}+\frac{1}{2}\sum\limits_{i=1}^{k}\operatorname{sign}_{i}-\frac{3}{4}\sigma(Q)-\frac{1}{2},$$
where $\operatorname{sign}_i$ denotes the sign of the contact surgery coefficient of $L_i$. 
\end{theorem} 

Note that in this special case, $Q$ is a symmetric matrix.
\bigskip


\section{Characterizing slopes for any Legendrian knot}
In this section, we prove the main result.
\begin{proof}[Proof of Theorem~\ref{Thm:anyknot}]
Suppose that contact $\frac{p}{q}$-surgeries along Legendrian knots $L$ and $L'$ in $(S^3,\xi_{st})$ yield contactomorphic manifolds, where $L$ is a Legendrian representative of $K$ and $L'$ is any Legendrian knot. Let $tb(L)=t$, $rot(L)=r$, $tb(L')=t'$ and $rot(L')=r'$. Since we only consider unoriented Legendrian knots, we assume that both $r$ and $r'$ are non-negative.

First, assume $\frac{p}{q}+tb(L)=0$. That is, the contact $(-t)$-surgeries along $L$ and $L'$ are contactomorphic. The first homology of $L(-t)$ is $\mathbb{Z}$, so the smooth surgery slope on $L'$ is $0$ and $t'=t$. So $L'$ is smoothly isotopic to $L$. Using Theorem~\ref{thm:d3} and the algorithm of Ding, Geiges and Stipsicz, the contact $(-t)$-surgery can be transformed to contact $\pm\frac{1}{n}$ surgeries. See the proof of Lemma~\ref{Lem:d3anyknotq=1}.  By Theorem~\ref{thm:d3}, we compute the Euler classes of the contact structures in $L(-t)$ and $L'(-t)$,  identified with elements in $H^2(S^3_0(L);\mathbb{Z})=\mathbb{Z}$ generated by the Poincaré-dual of the meridian:
    $$e(L(-t))=\left\{
    \begin{aligned}
       & \{r-t+1+2i:0\le i\le t-1\}, t\ge1,\\
    &r(t^2+t+1)\pm(t+1)^2, t\le-1,
    \end{aligned}
    \right.$$ $$
    e(L'(-t))=\left\{
    \begin{aligned}
       &\{r'-t+1+2i:0\le i\le t-1\}, t\ge1,\\
    &r'(t^2+t+1)\pm(t+1)^2, t\le-1.
    \end{aligned}
    \right.
    $$
Since $L(-t)$ and $L'(-t)$ have the same Euler classes,  we have $r=r'$.
\bigskip

Now we assume $\frac{p}{q}\neq -t$. The first homology groups of $L(\frac{p}{q})$ and $L'(\frac{p}{q})$ are $\mathbb{Z}_{p+qt}$ and $\mathbb{Z}_{p+qt'}$, respectively. These manifolds are homeomorphic only if either $t=t'$ or $t+t'=-\frac{2p}{q}$.
\bigskip

\subsection{Suppose $t'=t$.}\label{subs:anyknot1} Then their smooth surgery slopes must also be equal, i.e., $t+\frac{p}{q}=t'+\frac{p}{q}$. Since $\frac{p}{q}+t$ is assumed to be a characterizing slope for $K$, it follows that $L'$ is smoothly isotopic to $L$.  In what follows, we compute the $d_3$-invariant to prove $r=r'$, and consequently that $L$ and $L'$ are Legendrian isotopic in $(S^3, \xi_{st})$.

The contact $\frac{p}{q}$-surgery along $L$, $L(\frac{p}{q})$, corresponds to finitely many contact structures,  each of which can be obtained by converting the contact $\frac{p}{q}$-surgery diagram to a contact $(\pm1)$-surgery along a Legendrian link $L_1\sqcup L_2\sqcup\ldots\sqcup L_k$ according to the algorithm of Ding, Geiges and Stipsicz.   Here $L_{i+1}$ is either a Legendrian push-off or a stabilization of $L_i$. So each $L_i$ can be obtained from $L$  by Legendrian push-offs and stabilizations. Similarly, each contact structure in $L'(\frac{p}{q})$ can be represented by a contact $(\pm1)$-surgery along a Legendrian link $L'_1\sqcup L'_2\sqcup\ldots\sqcup L'_k$. Each $L'_i$ can be obtained from $L'$  by Legendrian push-offs and stabilizations. Note that $tb(L_i)=tb(L'_i)$, and the contact surgery coefficients of $L_i$ and $L'_i$ are the same. Let $r_i$ and $r'_i$ denote the rotation numbers of $L_i$ and $L'_i$, respectively, for $i=1,\dots,k$, and set $\boldsymbol{r}=(r_1,\dots,r_k)^T$ and $\boldsymbol{r}'=(r'_1,\dots,r'_k)^T$. Note that  $r_i-r$ ($i=1, 2, \ldots,k$) is a constant determined solely by the signs of stabilizations needed to obtain $L_i$ from $L$. There is a bijection between the expanded contact $(\pm1)$-surgeries for $L(\frac{p}{q})$ and those for $L'(\frac{p}{q})$: an expanded contact $(\pm1)$-surgery along $L_1\sqcup L_2\sqcup\ldots\sqcup L_k$ corresponds to an expanded contact $(\pm1)$-surgery along $L'_1\sqcup L'_2\sqcup\ldots\sqcup L'_k$ if and only if $r_i-r=r'_i-r'$ for each $i$. 

Suppose the contact structures in $L(\frac{p}{q})$ and $L'(\frac{p}{q})$ are contactomorphic. Then there exists a bijection between the sets of $d_3$-invariants of contact structures in $L(\frac{p}{q})$ and $L'(\frac{p}{q})$. However, the explicit correspondence is unknown. Therefore, we consider the sums of the $d_3$-invariants, denoted by $D_3(L, \frac{p}{q})$ and $D_3(L', \frac{p}{q})$.

By Theorem~\ref{thm:d3s},  for a contact $(\pm1)$-surgery along a Legendrian link $L_1\sqcup L_2\sqcup\ldots\sqcup L_k$, the $d_3$-invariant is given by  $$d_3=\frac{1}{4}\boldsymbol{r}^{T}Q^{-1}\boldsymbol{r}+\frac{1}{2}\sum\limits_{i=1}^{k}\operatorname{sign}_{i}-\frac{3}{4}\sigma(Q)-\frac{1}{2},$$
where $Q$ is the symmetric linking matrix, and $\operatorname{sign}_i$ denotes the sign of the contact surgery coefficient of $L_i$. This invariant can be viewed as a quadratic polynomial in the rotation number $r$ of $L$ with $p,q,t$ treated as constants. 
The quadratic term of $d_3$ is $$\frac{1}{4}(r,r,\ldots, r)Q^{-1}(r,r,\ldots, r)^{T}.$$ 
The linear term of $d_3$ is $$\frac{1}{2}(r_1-r,r_2-r,\ldots, r_k-r)Q^{-1}(r,r,\ldots, r)^{T}.$$ 
The constant term of $d_3$ is $$\frac{1}{4}(r_1-r,r_2-r,\ldots, r_k-r)Q^{-1}(r_1-r,r_2-r,\ldots, r_k-r)^{T}+\frac{1}{2}\sum\limits_{i=1}^{k}\operatorname{sign}_{i}-\frac{3}{4}\sigma(Q)-\frac{1}{2}.$$
When summing the $d_3$-invariants over all contact structures in $L(\frac{p}{q})$, we must exhaust all possible types of expanded contact $(\pm1)$-surgery diagrams. We observe that the linear term in $D_3(L, \frac{p}{q})$ vanishes, since reversing the signs of all stabilizations from $L$ to each $L_i$ yields a contact surgery diagram for which $\tilde{r}_i-r=-(r_i-r)$ for every $i\in\{1,\dots,k\}$.

For a contact $(\pm1)$-surgery along a Legendrian link $L'_1\sqcup L'_2\sqcup\cdots\sqcup L'_k$, since $t=t'$, the linking matrix remains $Q$, and the $d_3$-invariant is given by
$$d_3=\frac{1}{4}\boldsymbol{r}'^{T}Q^{-1}\boldsymbol{r}'+\frac{1}{2}\sum_{i=1}^{k}\operatorname{sign}_i-\frac{3}{4}\sigma(Q)-\frac{1}{2},$$
where $\operatorname{sign}_i$ denotes the sign of the contact surgery coefficient of $L'_i$. By the same reasoning as before, the linear term in $D_3(L',\frac{p}{q})$ also vanishes. 
Clearly, the constant terms of $D_3(L,\frac{p}{q})$ and $D_3(L',\frac{p}{q})$ coincide, and the leading coefficients of these two sums are also equal. Therefore, the equation $D_3(L,\frac{p}{q})=D_3(L',\frac{p}{q})$ admits a unique non-negative solution, namely $r'=r$. 
\bigskip

\subsection{Suppose $t+t'=-\frac{2p}{q}$.}\label{subs:trefoil2} Then either $q=1$ or $q=2$. By computing the $d_3$-invariants, we have the following two lemmas.

\begin{lemma}\label{Lem:d3anyknotq=1}
Assume $q=1$. If $6|t+p|$ cannot be expressed as a sum of two squares, then there exists no Legendrian knot $L'$ such that the contact $p$-surgeries along $L$ and along $L'$ yield contactomorphic contact manifolds.
\end{lemma}

\begin{proof} Note that $t+t'=-2p$ and $p$ is a non-zero integer. We divide the proof into four cases.

\textbf{Case 1}, $p\ge 2$. According to \cite{DingGeiges2004} and \cite{ding2004surgery}, the contact $p$-surgery along $L$ or $L'$ is equivalent to the surgery as shown in Figure~\ref{fig:+p}, whose surgery coefficients are $+1$ and $-\frac{1}{p-1}$.

\begin{figure}[htb]
\begin{overpic}
[scale=0.25]
{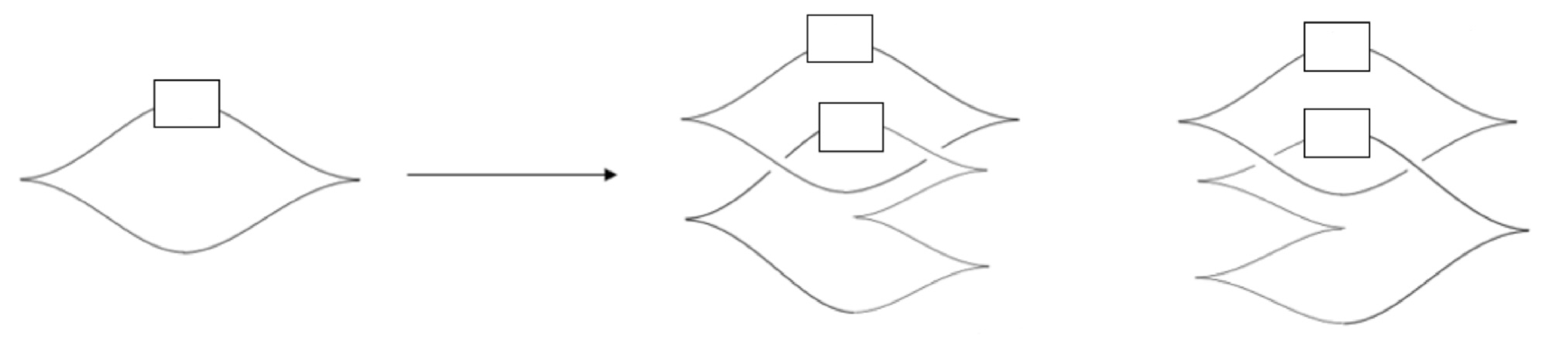} 
\put(56, 73){$L$}
\put(272, 67){$L$}
\put(269, 95){$L$}
\put(430, 65){$L$}
\put(430, 93){$L$}
\put(90, 70){$p$}
\put(300, 90){$+1$}
\put(300, 10){$-\frac{1}{p-1}$}
\put(470, 90){$+1$}
\put(470, 10){$-\frac{1}{p-1}$}
\end{overpic}
\caption{Contact $p$-surgery along $L$, where $p\ge2$.}
\label{fig:+p}
\end{figure}
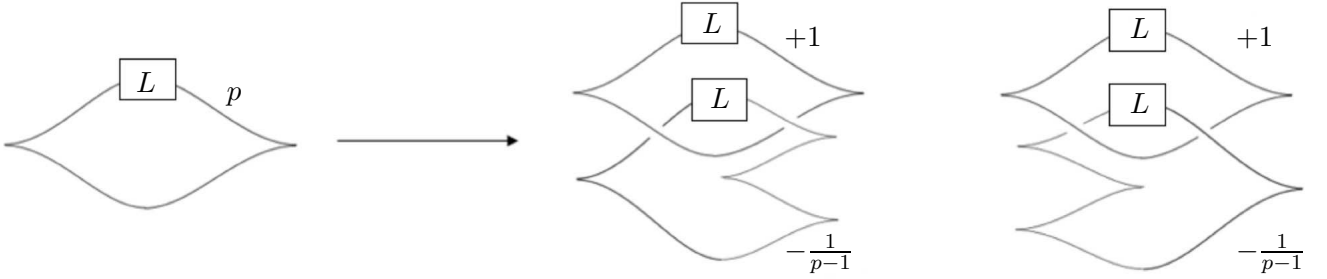

As to $L$, 
         $$\boldsymbol{r}=(r,r\pm 1)^{T},~Q=
\left(
\begin{matrix}
t+1&tp-t\\
t&tp-t-p
\end{matrix}
\right),~\boldsymbol{b}=(b_1,b_2)^T,$$

$$b_1=\frac{pr\pm pt\mp t}{t+p},~ b_2=\frac{-r\mp t\mp 1}{t+p}.$$

As to $L'$,
$$\boldsymbol{r'}=(r',r'\pm 1)^{T},~Q'=
\left(
\begin{matrix}
-2p-t+1&-2p^2-pt+2p+t\\
-2p-t&-2p^2-pt+p+t
\end{matrix}
\right),~\boldsymbol{b}'=(b'_1, b'_2)^T,$$

$$b'_1=\frac{-pr'\pm (2p+t)p\mp(2p+t)}{p+t},~
b'_2=\frac{r'\mp (2p+t)\pm 1}{p+t}.$$

If $t+p>0$,
$$\sigma(Q)=0,~\sigma (Q')=-2.$$ 
Then
$$d_3(\xi)=\frac{[r\pm (p-1)]^2}{4(t+p)}-\frac{3}{4},~d_3(\xi')=\frac{-[r'\pm (p-1)]^2}{4(t+p)}+\frac{3}{4}.$$
        
If $t+p<0$, 
$$\sigma (Q)=-2,~\sigma (Q')=0.$$ 
Then 
$$d_3(\xi)=\frac{[r\pm (p-1)]^2}{4(t+p)}+\frac{3}{4},~d_3(\xi')=\frac{-[r'\pm (p-1)]^2}{4(t+p)}-\frac{3}{4}.$$

In the above equations, $\pm$ correspond to the two different stabilizations as shown in Figure~\ref{fig:+p}. So in any cases, $d_3(\xi)=d_3(\xi')$ if and only if 
$$[r\pm (p-1)]^2+[r'\pm (p-1)]^2=6|t+p|.$$
Regardless of how the signs are chosen, $6|t+p|$ can be written as the sum of two squares.

\textbf{Case 2}, $p=1$. By the assumption, $t\neq -1$. 
If $t<-1$, $$Q=t+1<0,~\sigma(Q)=-1,~b=\frac{r}{t+1}; ~Q'=-t-1>0,~\sigma(Q')=+1,~b=-\frac{r'}{t+1},$$
then
$$d_3(\xi)=\frac{r^2}{4(1+t)}+\frac{3}{4}, ~d_3(\xi')=\frac{-r'^2}{4(1+t)}-\frac{3}{4}.$$
If $t>-1$, $$Q=t+1>0, ~\sigma(Q)=1, ~b=\frac{r}{t+1}; ~Q'=-t-1<0, ~\sigma(Q')=-1, ~b=-\frac{r'}{t+1},$$
then
$$d_3(\xi)=\frac{r^2}{4(1+t)}-\frac{3}{4},~d_3(\xi')=\frac{-r'^2}{4(1+t)}+\frac{3}{4}.$$
In the both cases, it can be deduced from $d_3(\xi)=d_3(\xi')$ that
$$r^2+r'^2=6|1-t|.$$

\textbf{Case 3},  $p=-1$. Then $t\neq 1$ and $t'=2-t$. If $t\le0$,
$$Q=t-1>0, ~\sigma(Q)=1, ~b=\frac{r}{t-1}; ~Q'=1-t<0, ~\sigma(Q')=-1, ~b'=\frac{r'}{1-t},$$
then
$$d_3(\xi)=\frac{-r^2}{4(1-t)}-\frac{7}{4}, ~
d_3(\xi')=\frac{r'^2}{4(1-t)}-\frac{1}{4}.$$
If $t\ge2$,
$$Q=t-1<0, ~\sigma(Q)=-1, ~b=\frac{r}{t-1}; ~Q'=1-t>0, ~\sigma(Q')=1, ~b'=\frac{r'}{1-t},$$
then
$$d_3(\xi)=\frac{-r^2}{4(1-t)}-\frac{1}{4}, ~
d_3(\xi')=\frac{r'^2}{4(1-t)}-\frac{7}{4}.$$
For either $t\leq0$ or $t\geq2$, it can be deduced from $d_3(\xi)=d_3(\xi')$ that
$$r^2+r'^2=6|1-t|.$$

\textbf{Case 4}, $p\le-2$. The contact $p$-surgery along $L$ (or $L'$) is equivalent to the contact $(-1)$-surgery along the $-p-1$ stabilizations of $L$ (or $L'$), whose Thurston-Bennequin invariant is $t+p+1$. By Case 3, $6|1-(t+p+1)|=6|t+p|$ is a sum of two squares. 
\end{proof}

\begin{lemma}\label{Lem:d3anyknotq=2}
Assume $q=2$. If $3|2t+p|$ cannot be expressed as a sum of two squares, then there exists no Legendrian knot $L'$ such that the contact $\frac{p}{2}$-surgeries along $L$ and along $L'$ yield contactomorphic contact manifolds.
\end{lemma}

\begin{proof} Note that $t+t'=-p$ and $p$ is an odd integer. We divide the proof into four cases.


\begin{figure}[htb]
\begin{overpic}
[scale=0.5]
{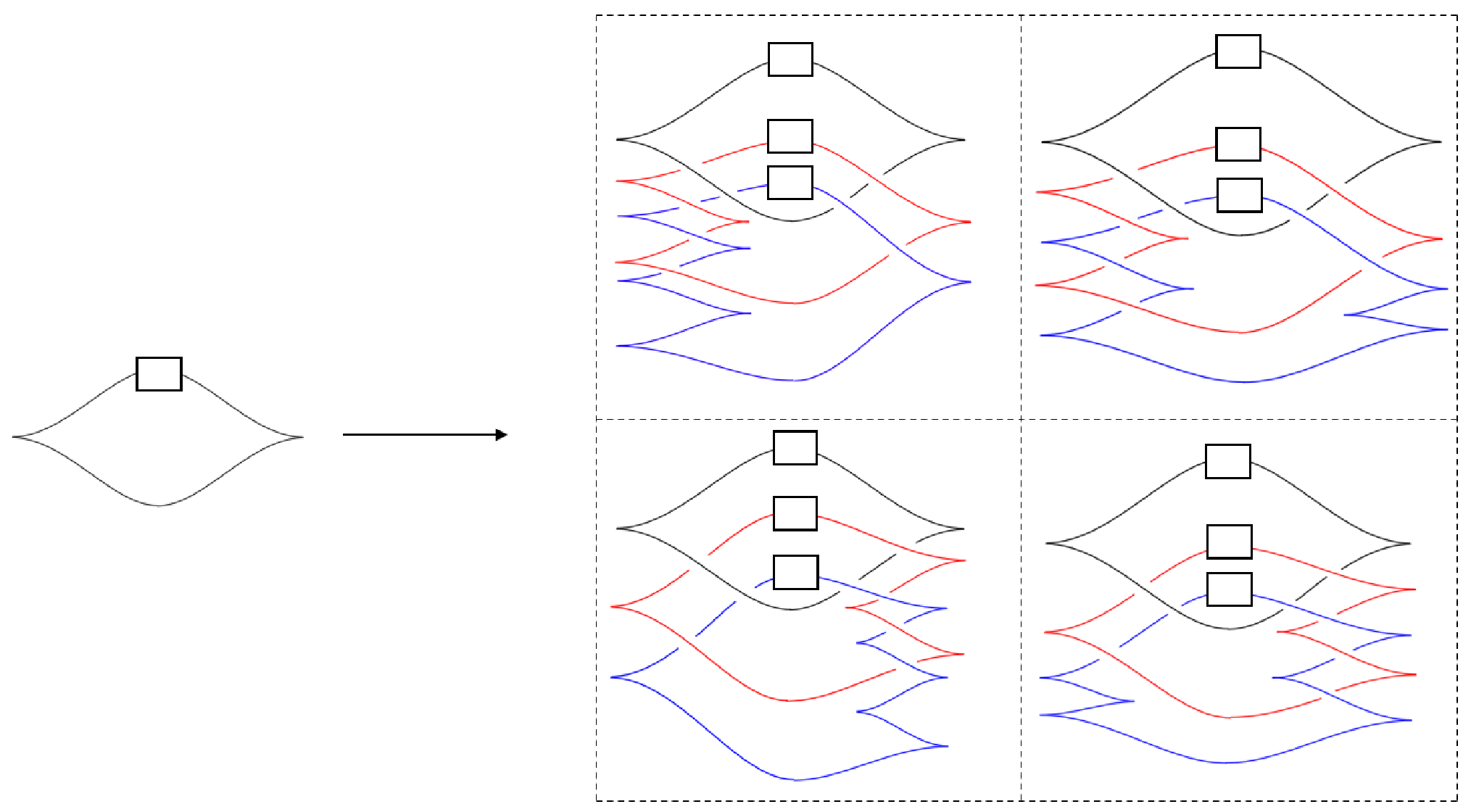} 
\put(49, 142){$L$}
\put(259, 246){$L$}
\put(259, 221){$L$}
\put(259, 205){$L$}
\put(260, 117){$L$}
\put(260, 95){$L$}
\put(260, 75){$L$}
\put(408, 249){$L$}
\put(408, 218){$L$}
\put(408, 201){$L$}
\put(405, 113){$L$}
\put(405, 86){$L$}
\put(405, 70){$L$}
\put(80, 140){$\frac{p}{2}$}
\put(310, 230){$+1$}
\put(310, 156){$-1$}
\put(310, 98){$+1$}
\put(310, 30){$-1$}
\put(465, 230){$+1$}
\put(455, 144){$-1$}
\put(460, 95){$+1$}
\put(455, 19){$-1$}
\put(310, 205){$-\frac{1}{\frac{p-3}{2}}$}
\put(310, 73){$-\frac{1}{\frac{p-3}{2}}$}
\put(455, 203){$-\frac{1}{\frac{p-3}{2}}$}
\put(456, 66){$-\frac{1}{\frac{p-3}{2}}$}
\put(200,134){$(1)$}
\put(340,134){$(2)$}
\put(200,7){$(3)$}
\put(340,7){$(4)$}
\end{overpic}
\caption{Contact $\frac{p}{2}$-surgery along $L$, where $p\geq 5$.}
\label{fig:+p;2}
\end{figure}
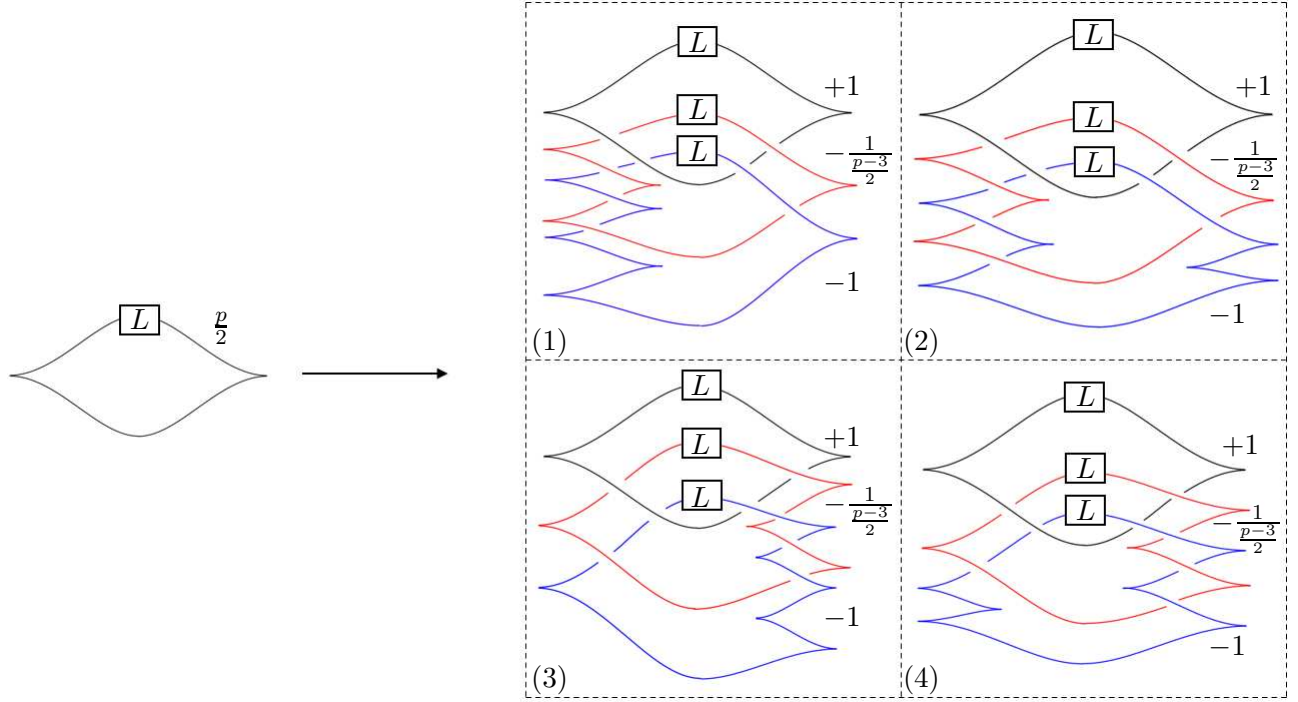

\textbf{Case 1}, $p\ge5$. Note that $\frac{p}{2-p}$ has the continued fraction expansion 
        $$\frac{p}{2-p}=[\underbrace{-2,\dots ,-2}_{\frac{p-3}{2}},-3].$$
The contact $\frac{p}{2}$-surgery along $L$ (or $L'$) is equivalent to the contact surgeries along the four Legendrian links in Figure~\ref{fig:+p;2}. We compute the $d_3$-invariants using Theorem~\ref{thm:d3}.  

Since $tb(L)=t$ and $tb(L')=-t-p$, the generalized linking matrices are
$$Q=\begin{pmatrix}
	t+1&\frac{pt-3t}{2}&t\\
	t&\frac{pt-3t-p+1}{2}&t-1\\
    t&\frac{pt-3t-p+3}{2}&t-3
	\end{pmatrix}~\text{and}~
    Q'=\begin{pmatrix}
	-t-p+1&\frac{-p^2-pt+3p+3t}{2} &-t-p \\
	-t-p&\frac{-p^2-pt+2p+3t+1}{2} &-t-p-1 \\
    -t-p&\frac{-p^2-pt+2p+3t+3}{2} &-t-p-3 
	\end{pmatrix}.
    $$
    Note that $\det(Q)=p+2t$ and $\det(Q')=-p-2t$, and the eigenvalues of $Q$ and $Q'$ are real.

In the expanded surgery diagrams,  $\boldsymbol{r}$ has the following four cases:
$$\boldsymbol{r}^1=(r,r+1,r+2)^T,~\boldsymbol{r}^2=(r,r+1,r)^T,~\boldsymbol{r}^3=(r,r-1,r-2)^T,~\boldsymbol{r}^4=(r,r-1,r)^T,$$
and  $\boldsymbol{r}'$ has the following four cases:
$$\boldsymbol{r}'^{1}=(r',r'+1,r'+2)^T,~\boldsymbol{r}'^{2}=(r',r'+1,r')^T,~\boldsymbol{r}'^{3}=(r',r'-1,r'-2)^T,~\boldsymbol{r}'^{4}=(r',r'-1,r')^T.$$


We solve the equations $Q\boldsymbol{b}=\boldsymbol{r}$ and $Q'\boldsymbol{b}'=\boldsymbol{r}'$ in each case. 

$$\boldsymbol{b}^1=(\frac{pr+pt-t}{p+2t},-\frac{2r+2t+1}{p+2t},-\frac{r+2t+\frac{p+1}{2}}{p+2t})^T,$$
$$\boldsymbol{b}^2=(\frac{pr+pt-3t}{p+2t},-\frac{2r+2t+3}{p+2t},\frac{-r+\frac{p-3}{2}}{p+2t})^T,$$
$$\boldsymbol{b}^3=(\frac{pr-pt+t}{p+2t},\frac{-2r+2t+1}{p+2t},\frac{-r+2t+\frac{p+1}{2}}{p+2t})^T,$$
$$\boldsymbol{b}^4=(\frac{pr-pt+3t}{p+2t},\frac{-2r+2t+3}{p+2t},\frac{-r-\frac{p-3}{2}}{p+2t})^T;$$

$$\boldsymbol{b}'^{1}=(\frac{p^2-pr'+pt-p-t}{p+2t},\frac{2r'-2p-2t+1}{p+2t},\frac{r'-2t+\frac{1-3p}{2}}{p+2t})^T,$$
$$\boldsymbol{b}'^{2}=(\frac{p^2-pr'+pt-3p-3t}{p+2t},\frac{2r'-2p-2t+3}{p+2t},\frac{r'-\frac{p-3}{2}}{p+2t})^T,$$
$$\boldsymbol{b}'^{3}=(-\frac{p^2+pr'+pt-p-t}{p+2t},\frac{2r'+2p+2t-1}{p+2t},\frac{r'+2t+\frac{3p-1}{2}}{p+2t})^T,$$
$$\boldsymbol{b}'^{4}=(-\frac{p^2+pr'+pt-3p-3t}{p+2t},\frac{2r'+2p+2t-3}{p+2t},\frac{r'+\frac{p-3}{2}}{p+2t})^T.$$

When $p+2t>0$, the determinant $\det(Q)=p+2t>0$, $\det(Q')=-p-2t<0$. So for any $t$, the eigenvalues of $Q$ and $Q'$ are non-zero. From continuity, the sign of the three eigenvalues keep unchanged, and when $p=5,~t=0$, we can obtain that $\sigma(Q)=-1$, $\sigma(Q')=-3$. So for any $p$ and $t$, $\sigma(Q)=-1$, $\sigma(Q')=-3$.  Then the $d_3$-invariants can be computed that 
    $$d_3(\xi^1)=\frac{(2r-p+1)^2}{8(p+2t)}-1, ~d_3(\xi^2)=\frac{(2r-p+3)^2}{8(p+2t)}-\frac{1}{2},$$
    $$d_3(\xi^3)=\frac{(-2r-p+1)^2}{8(p+2t)}-1, ~d_3(\xi^4)=\frac{(-2r-p+3)^2}{8(p+2t)}-\frac{1}{2},$$
    $$d_3(\xi'^1)=-\frac{(2r'-p+1)^2}{8(p+2t)}+\frac{1}{2},~ d_3(\xi'^2)=-\frac{(2r'-p+3)^2}{8(p+2t)}+1,$$
    $$d_3(\xi'^3)=-\frac{(-2r'-p+1)^2}{8(p+2t)}+\frac{1}{2},~ d_3(\xi'^4)=-\frac{(-2r'-p+3)^2}{8(p+2t)}+1.$$
    
When $p+2t<0$, the determinant $\det(Q)=p+2t<0$, $\det(Q')=-p-2t>0$. So for any $t$, the eigenvalues of $Q$ and $Q'$ are non-zero. From continuity, the sign of the three eigenvalues keep unchanged, and when $p=5,~t=-3$, we can obtain that $\sigma(Q)=-3$, $\sigma(Q')=-1$. So for any $p$ and $t$, $\sigma(Q)=-3$, $\sigma(Q')=-1$.
Then the $d_3$-invariants can be computed that 
    $$d_3(\xi^1)=\frac{(2r-p+1)^2}{8(p+2t)}+\frac{1}{2},~ d_3(\xi^2)=\frac{(2r-p+3)^2}{8(p+2t)}+1,$$
    $$d_3(\xi^3)=\frac{(-2r-p+1)^2}{8(p+2t)}+\frac{1}{2},~ d_3(\xi^4)=\frac{(-2r-p+3)^2}{8(p+2t)}+1,$$
    $$d_3(\xi'^1)=-\frac{(2r'-p+1)^2}{8(p+2t)}-1, ~d_3(\xi'^2)=-\frac{(2r'-p+3)^2}{8(p+2t)}-\frac{1}{2},$$
    $$d_3(\xi'^3)=-\frac{(-2r'-p+1)^2}{8(p+2t)}-1,~ d_3(\xi'^4)=-\frac{(-2r'-p+3)^2}{8(p+2t)}-\frac{1}{2}.$$
If the contact manifolds obtained by contact $\frac{p}{2}$-surgery along $L$ and along $L'$ are contactomorphic, then  $$\{d_3(\xi^1), d_3(\xi^2),d_3(\xi^3),d_3(\xi^4)\}=\{d_3(\xi'^1),  d_3(\xi'^2), d_3(\xi'^3), d_3(\xi'^4)\}.$$ For either $p+2t>0$ or $p+2t<0$, since $3|p+2t|$ cannot be expressed as a sum of two squares, Theorem~\ref{notwosquares} implies $12|p+2t|$ is also not a sum of two squares. So the sets $\{d_3(\xi^1), d_3(\xi^3)\}$ and $\{ d_3(\xi'^1), d_3(\xi'^3)\}$ are disjoint, as are $\{d_3(\xi^2), d_3(\xi^4)\}$ and $\{ d_3(\xi'^2), d_3(\xi'^4)\}$.  Thus we have
$$\{ d_3(\xi^1),d_3(\xi^3)\}=\{d_3(\xi'^2),d_3(\xi'^4)\} ~\text{and}~
\{ d_3(\xi^2),d_3(\xi^4)\}=\{d_3(\xi'^1),d_3(\xi'^3)\}.$$
However, $d_3(\xi^1)+d_3(\xi^3)=d_3(\xi'^2)+d_3(\xi'^4)$ means that
$$r^2+r'^2+(\frac{p-1}{2})^2+(\frac{p-3}{2})^2=4|p+2t|,$$
and $d_3(\xi^2)+d_3(\xi^4)=d_3(\xi'^1)+d_3(\xi'^3)$ means that
$$r^2+r'^2+(\frac{p-1}{2})^2+(\frac{p-3}{2})^2=2|p+2t|,$$
which leads to a contradiction. Thus, the contact structures in $L(\frac{p}{2})$ and those in $L'(\frac{p}{2})$ do not have the same $d_3$-invariants, so they are not contactomorphic.

\textbf{Case 2},  $p=3$. The contact $\frac{3}{2}$-surgery along $L$ (or $L'$) is equivalent to the contact surgery along the Legendrian link shown in Figure~\ref{fig:+3/2}.
\begin{figure}[htb]
\begin{overpic}
[scale=0.5]
{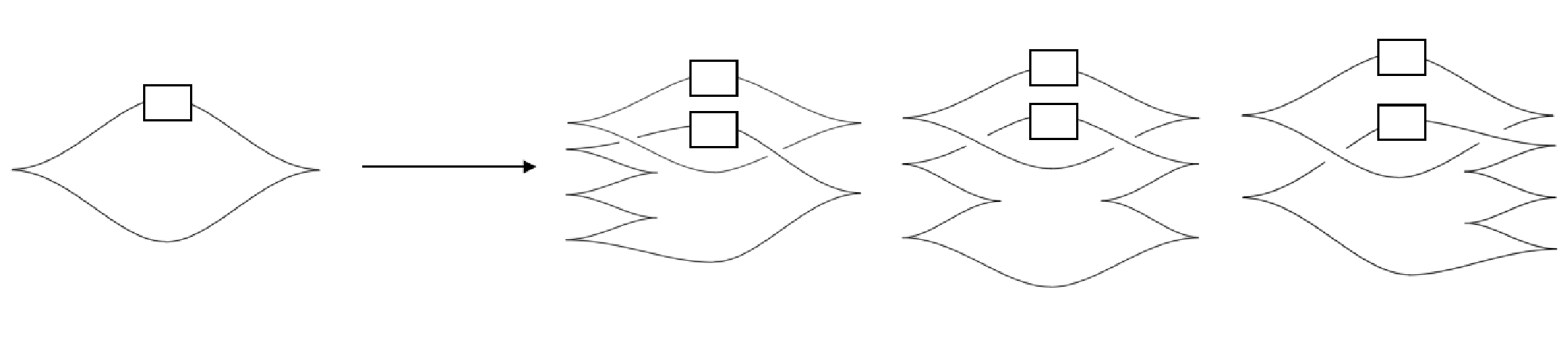} 
\put(50, 73){$L$}
\put(226, 81){$L$}
\put(226, 64){$L$}
\put(336, 67){$L$}
\put(336, 84){$L$}
\put(449, 88){$L$}
\put(449, 67){$L$}
\put(72, 77){$+\frac{3}{2}$}
\put(255, 81){$+1$}
\put(255, 27){$-1$}
\put(375, 80){$+1$}
\put(375, 20){$-1$}
\put(494, 80){$+1$}
\put(494, 15){$-1$}
\end{overpic}
\caption{Contact $\frac{3}{2}$-surgery along $L$.}
\label{fig:+3/2}
\end{figure}
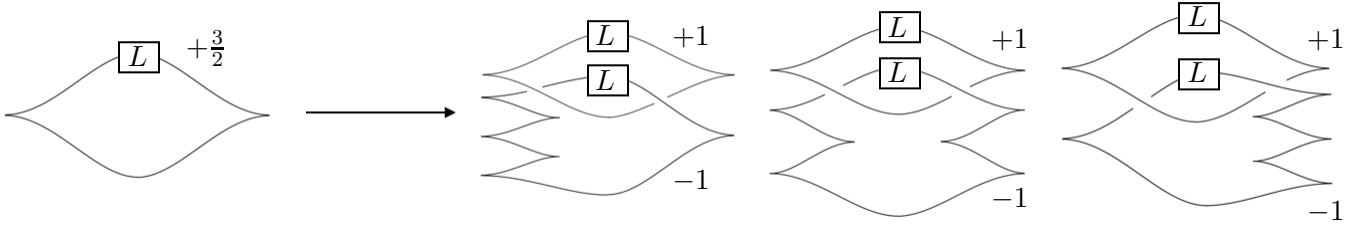
    Then 
    $$\boldsymbol{t}=\begin{pmatrix}
	t\\
	t-2
	\end{pmatrix},~
    Q=\begin{pmatrix}
	t+1&t\\
	t&t-3
	\end{pmatrix},~
    \boldsymbol{r}=\begin{pmatrix}
	r\\
	r_2
	\end{pmatrix},r_2=r+2,r,r-2.$$
    $$\boldsymbol{t'}=\begin{pmatrix}
	-3-t\\
	-5-t
	\end{pmatrix},~
    Q'=\begin{pmatrix}
	-2-t&-3-t\\
	-3-t&-6-t
	\end{pmatrix},~
    \boldsymbol{r'}=\begin{pmatrix}
	r'\\
	r_2'
	\end{pmatrix},r_2'=r'+2,r',r'-2.$$
Note that $\det(Q)=-2t-3$, $\text{tr}(Q)=2t-2$, $\det(Q')=2t+3$, $\text{tr}(Q')=-2t-8$.
    We solve the equations $Q\boldsymbol{b}=\boldsymbol{r}$ and $Q'\boldsymbol{b}'=\boldsymbol{r}'$ in each case.
    $$\boldsymbol{b}^1=(\frac{2t+3r}{2t+3},-\frac{2t+r+2}{2t+3})^T,~\boldsymbol{b}^2=(\frac{3r}{2t+3},\frac{-r}{2t+3})^T,~\boldsymbol{b}^3=(\frac{-2t+3r}{2t+3},\frac{2t-r+2}{2t+3})^T;$$
$$\boldsymbol{b}'^{1}=(\frac{2t-3r+6}{2t+3},\frac{-2t+r-4}{2t+3})^T,~\boldsymbol{b}'^{2}=(\frac{-3r}{2t+3},\frac{r}{2t+3})^T,~\boldsymbol{b}'^{3}=(\frac{-2t-3r-6}{2t+3},\frac{2t+r+4}{2t+3})^T.
$$
    
When $t\ge -1$, it follows from the determinants and traces that $\sigma(Q)=0$, $\sigma(Q')=-2$. So
    $$d_3(\xi^1)=\frac{(r-1)^2}{2(2t+3)}-1,
    ~d_3(\xi^2)=\frac{r^2}{2(2t+3)}-\frac{1}{2},
    ~d_3(\xi^3)=\frac{(r+1)^2}{2(2t+3)}-1, $$
    $$d_3(\xi'^1)=\frac{-(r'-1)^2}{2(2t+3)}+\frac{1}{2}, ~d_3(\xi'^2)=\frac{-r'^2}{2(2t+3)}+1,
    ~d_3(\xi'^3)=\frac{-(r'+1)^2}{2(2t+3)}+\frac{1}{2}.$$ 

When $t\le-2$, $\sigma(Q)=-2$, $\sigma(Q')=0$. So
    $$d_3(\xi^1)=\frac{(r-1)^2}{2(2t+3)}+\frac{1}{2}, ~ 
    d_3(\xi^2)=\frac{r^2}{2(2t+3)}+1,~
    d_3(\xi^3)=\frac{(r+1)^2}{2(2t+3)}+\frac{1}{2}, ~ $$
    $$d_3(\xi'^1)=\frac{-(r'-1)^2}{2(2t+3)}-1, ~ 
    d_3(\xi'^2)=\frac{-r'^2}{2(2t+3)}-\frac{1}{2}, ~
    d_3(\xi'^3)=\frac{-(r'+1)^2}{2(2t+3)}-1. $$
For either $t\geq-1$ or $t\leq-2$, if $3|2t+3|$ cannot be expressed as a sum of two squares, then $$\{d_3(\xi^1), d_3(\xi^2),d_3(\xi^3)\}\neq\{d_3(\xi'^1),  d_3(\xi'^2), d_3(\xi'^3)\}.$$ So the contact structures are not contactomorphic. 

\textbf{Case 3}, $p=1$. The contact $\frac{1}{2}$-surgery along $L$ (or $L'$) is equivalent to the $(+1)$-surgeries along two push-offs of $L$ (or $L'$). Then
     \begin{align*}
	Q=\begin{pmatrix}
	t+1&t\\
	t&t+1
	\end{pmatrix},
    \boldsymbol{r}=\begin{pmatrix}
	r\\
	r
	\end{pmatrix},
    \boldsymbol{b}=\begin{pmatrix}
	\frac{r}{2t+1}\\
	\frac{r}{2t+1}
	\end{pmatrix};~
    Q'=\begin{pmatrix}
	-t&-t-1\\
	-t-1&-t
	\end{pmatrix},
    \boldsymbol{r'}=\begin{pmatrix}
	r'\\
	r'
	\end{pmatrix},
    \boldsymbol{b'}=\begin{pmatrix}
	-\frac{r'}{2t+1}\\
	-\frac{r'}{2t+1}
	\end{pmatrix}.
    \end{align*}
The two eigenvalues of $Q$ is 1 and $2t+1$, and the two eigenvalues of $Q'$ is 1 and $-1-2t$. If $t\ge0$, then $\sigma(Q)=2$ and $\sigma(Q')=0$. It follows that $d_3(\xi)=\frac{r^2}{2(2t+1)}-1$ and $d_3(\xi')=\frac{-r'^2}{2(2t+1)}+\frac{1}{2}$. So $d_3(\xi)=d_3(\xi')$ implies $r^2+r'^2=3(2t+1)$. If $t\le -1$, then $\sigma(Q)=0$ and $\sigma(Q')=2$.  It follows that $d_3(\xi)=\frac{r^2}{2(2t+1)}+\frac{1}{2}$ and $d_3(\xi')=\frac{-r'^2}{2(2t+1)}-1$. So $d_3(\xi)=d_3(\xi')$ implies $r^2+r'^2=-3(2t+1)$. Thus, $3|2t+1|$ can be expressed as a sum of two squares.

\textbf{Case 4},  $p\le-1$. Since $\frac{p}{2}=[\frac{p-1}{2}, -2]$, the contact $\frac{p}{2}$-surgery along $L$ (or $L'$) is equivalent to the contact surgery along the Legendrian link shown in Figure~\ref{fig:-p/2}, where $L_1=S_+^iS_-^{-\frac{p+1}{2}-i}(L),i=0,1,\ldots,-\frac{p+1}{2}$. Then the rotation number of $L_1$ is $r_i=r+\frac{p+1}{2}+2i,i=0,1,\ldots,-\frac{p+1}{2}$. 

\begin{figure}[htb]
\begin{overpic}
[scale=0.5]
{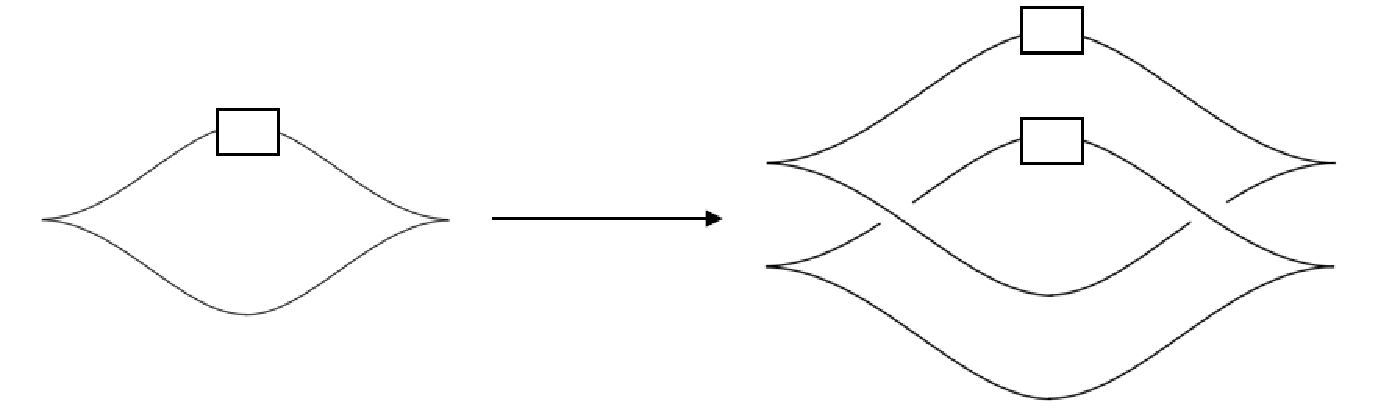} 
\put(56,65){$L$}
\put(249,90){$L_1$}
\put(249,63){$L_1$}
\put(90,60){$\frac{p}{2}$}
\put(300,70){$-1$}
\put(300,20){$-1$}
\end{overpic}
\caption{Contact $\frac{p}{2}$-surgery along $L$, where $p\le-1$.}
\label{fig:-p/2}
\end{figure}


Then
     \begin{align*}
	Q=\begin{pmatrix}
	t+\frac{p-1}{2}&t+\frac{p+1}{2}\\
	t+\frac{p+1}{2}&t+\frac{p-1}{2}
	\end{pmatrix},
    \boldsymbol{r}=\begin{pmatrix}
	r_i\\
	r_i
	\end{pmatrix},
    \boldsymbol{b}=\begin{pmatrix}
	\frac{r_i}{p+2t}\\
	\frac{r_i}{p+2t}
	\end{pmatrix},
    \end{align*}
    \begin{align*}
    Q'=\begin{pmatrix}
	-t-\frac{p+1}{2}&-t-\frac{p-1}{2}\\
	-t-\frac{p-1}{2}&-t-\frac{p+1}{2}
	\end{pmatrix},
    \boldsymbol{r'}=\begin{pmatrix}
	r_j\\
	r_j
	\end{pmatrix},
    \boldsymbol{b'}=\begin{pmatrix}
	\frac{-r_j}{p+2t}\\
	\frac{-r_j}{p+2t}
	\end{pmatrix}.
    \end{align*}
We have  $\det(Q)=-(p+2t)$ and $\det(Q')=p+2t$. When $p+2t>0$, $\sigma(Q)=0$ and $\sigma(Q')=-2$. So
$$d_3(\xi^i)=\frac{r_i^2}{2(p+2t)}-\frac{3}{2}, ~d_3(\xi'^j)=-\frac{r_j^2}{2(p+2t)}, ~i,j\in\{0,1,\ldots,-\frac{p+1}{2}\}.$$
When $p+2t<0$, $\sigma(Q)=-2$ and $\sigma(Q')=0$. So
$$d_3(\xi^i)=\frac{r_i^2}{2(p+2t)}, ~d_3(\xi'^j)=-\frac{r_j^2}{2(p+2t)}-\frac{3}{2}, ~i,j\in\{0,1,\ldots,-\frac{p+1}{2}\}.$$
In both cases, if $3|p+2t|$ cannot be expressed as a sum of two squares, $d_3(\xi^i)\neq d_3(\xi'^j)$. Thus, the contact structure in the two cases do not have the same $d_3$-invariants, so they are not contactomorphic.
\end{proof}

Suppose $L(\frac{p}{q})$ and $L'(\frac{p}{q})$ are contactomorphic. Then the manifolds $S^3_{\frac{p}{q}+t}(L)$ and $S^3_{-\frac{p}{q}-t}(L')$ are orientation-preserving homeomorphic.
For the case $q=1$, Proposition~\ref{Prop:surcharanyknot} implies that $|p+t|$ has no prime factor of the form $4k+3$; in particular, $3\nmid |p+t|$. Lemma~\ref{Lem:d3anyknotq=1} then asserts that $6|p+t|$ is a sum of two squares. However, since the factor $3$ arises solely from the coefficient $6$, the integer $6|p+t|$ contains exactly one factor of $3$. By Theorem~\ref{notwosquares}, an integer in which the prime $3$ occurs to an odd power cannot be expressed as a sum of two squares, yielding a contradiction. When $q=2$, Proposition~\ref{Prop:sucharanyknot2} gives $|p+2t|=4l+1$. Lemma~\ref{Lem:d3anyknotq=2} then implies that $3|p+2t|=12l+3$ is a sum of two squares. But every square is congruent to $0$ or $1$ modulo $4$, so a sum of two squares is congruent to $0,1$, or $2$ modulo $4$. Yet $12l+3\equiv 3\pmod 4$, which is impossible.

Thus, both cases lead to contradictions.
\end{proof}

In the following, we prove the corollaries of Theorem~\ref{Thm:anyknot}.


\begin{proof}[Proof of Corollary~\ref{Thm:unknot}]
According to \cite{ElFr} and \cite{EtHo01}, the unknot, right handed trefoil, left handed trefoil and figure eight knot are all Legendrian simple. By Theorem~\ref{Thm:surcharunknot} and Theorem~\ref{Thm:surchartrefoil}, all rational numbers are characterizing for these knots. So the corollary follows directly from Corollary~\ref{Thm:anyknot1}.  
\end{proof}

\begin{remark}
In \cite[Example 1.9]{CEK}, Casals, Etnyre and Kegel state the contact $(+6)$-surgeries along a Legendrian unknot with $tb=-11$ and $rot=0$, and a Legendrian right handed trefoil with $tb=-1$ and $rot=0$, are contactomorphic. However, since $6|t+p|=30$ is not the sum of two squares, this is not true by Lemma~\ref{Lem:d3anyknotq=1}. 
\end{remark}

\begin{proof}[Proof of Corollary~\ref{Thm:52}]
According to \cite{EtNgVe13}, the knot $5_2$ is Legendrian simple with maximal Thurston-Bennequin invariant $-8$. The knot $\overline{5_2}$ is not Legendrian simple. It has maximal Thurston-Bennequin invariant $1$. A Legendrian $\overline{5_2}$ is determined by its Thurston-Bennequin invariant and rotation number unless it has $tb=1$. By Theorem~\ref{Thm:surchar52}, if $r$ is a rational number other than a positive integer, then $r$ is characterizing for $5_2$.  So, if $r$ is a rational number other than a negative integer, then $r$ is characterizing for  $\overline{5_2}$. Hence the parts (1) and (2) follow directly from Corollary~\ref{Thm:anyknot1}. 

By \cite{EtNgVe13}, there are exactly two Legendrian representatives of $\overline{5_2}$ with maximal Thurston-Bennequin invariant $1$. We denote them by $L$ and $\tilde{L}$. Note that $0$ is characterizing for $\overline{5_2}$. By Theorem~\ref{Thm:anyknot}, if $L(-1)$ is contactomorphic to $L'(-1)$ for some Legendrian knot $L'$ in $(S^3, \xi_{st})$, then $L'$ is a Legednrian representative of $\overline{5_2}$ with $tb=1$. On the other hand,  according to \cite[Section 7.3]{BEE12},  $L(-1)$ and $\tilde{L}(-1)$ are not contactomorphic. So $L'$ is isotopic to $L$. The part (3) also holds.
\end{proof}

\begin{proof}[Proof of Corollary~\ref{Thm:T52}]
According to \cite{EtHo01}, the torus knot $T_{5,2}$ and $T_{5,-2}$ are both Legendrian simple.  By Theorem~\ref{Thm:surcharT52}, suppose $r>-1$ is a rational number, and $r\notin\{0, 1, \pm\frac{1}{2}, \pm\frac{1}{3}\}$, then $r$ is a characterizing slope for $T_{5,2}$. So
if $r<1$ is a rational number, and $r\notin\{0, -1, \pm\frac{1}{2}, \pm\frac{1}{3}\}$, then $r$ is a characterizing slope for $T_{5,-2}$. Thus the corollary follows from Corollary~\ref{Thm:anyknot1}.
\end{proof}
\bigskip

\bibliographystyle{alpha} 
\bibliography{refs}
\Addresses
\end{document}